\documentclass{amsart}

\usepackage{amsmath,amssymb,amsthm}
\usepackage{hyperref}

\newtheorem{theorem}{Theorem}[section]

\newtheorem{prop}[theorem]{Proposition}
\newtheorem{cor}[theorem]{Corollary}

\theoremstyle{definition}
\newtheorem{definition}[theorem]{Definition}
\newtheorem{remark}[theorem]{Remark}
\newtheorem{problem}[theorem]{Problem}
\newtheorem{example}[theorem]{Example}

\title[A residually finite analogue of Kegel's theorem]
{A residually finite analogue of Kegel's theorem on splitting automorphisms}

\author{Alfonso DI BARTOLO}
\address{Dipartimento di Matematica e Informatica, Università degli Studi di Palermo, Via Archirafi 34, 90123 Palermo, Italy}
\email{alfonso.dibartolo@unipa.it}

\author{Kıvanç ERSOY}
\address{Institut für Mathematik, Freie Universität Berlin, Arnimallee 7, 14195 Berlin, Germany}
\email{ersoy@zedat.fu-berlin.de}

\author{Giovanni FALCONE}
\address{Dipartimento di Matematica e Informatica, Università degli Studi di Palermo, Via Archirafi 34, 90123 Palermo, Italy}
\email{giovanni.falcone@unipa.it}


\begin{document}

\begin{abstract}
Thompson proved that every finite group admitting a fixed-point-free automorphism of prime order is nilpotent, and Kegel showed that the same conclusion holds for finite groups admitting a splitting automorphism of prime order. Motivated by these results, Sozutov asked whether a \(p'\)-group admitting a splitting automorphism of prime order is locally nilpotent if
\[
\langle g, g^\varphi, \dots, g^{\varphi^{p-1}} \rangle
\]
is nilpotent for every \(g \in G\) \cite[Problem 10.59]{kourovka21}. We prove that if \(G\) is a periodic residually finite $p'$-group admitting a splitting automorphism of prime order \(p\) then \(G\) is nilpotent of class bounded in terms of \(p\). This gives an affirmative answer, for residually finite groups, to the problem of Sozutov. We also prove that a possible counterexample to Sozutov's problem cannot be a Tarski monster.
\end{abstract}

\maketitle

\section{Introduction}

Let \(G\) be a group and let \(\alpha \in \operatorname{Aut}(G)\). An automorphism \(\alpha\) is called \emph{fixed-point-free} if \(C_G(\alpha)=1\). For finite groups, every fixed-point-free automorphism \(\alpha\) of order \(n\) satisfies
\[
x x^\alpha x^{\alpha^2}\cdots x^{\alpha^{n-1}}=1
\]
for all \(x\in G\) (see \cite[10.5.1]{Robinson}). This led Gorchakov to introduce the notion of a splitting automorphism.

\begin{definition}
Let \(G\) be a group. An automorphism \(\alpha\) of order \(n\) is called a \emph{splitting automorphism} if
\[
g g^\alpha g^{\alpha^2}\cdots g^{\alpha^{n-1}}=1
\]
for every \(g\in G\).
\end{definition}

An element \(a \in G\) is called an \textbf{anticentral element} of \(G\) if
\[
aG' = a^G
\]
(see \cite{Ladisch2008} or \cite{Ersoy12}). It is easy to see that if \(a \in G\) is anticentral, then conjugation by \(a\) induces a splitting automorphism of \(G'\). Moreover, if \(\alpha\) is a fixed-point-free automorphism of a finite group \(G\), then \(\alpha\) is an anticentral element of the semidirect product
\[
H = G \rtimes \langle \alpha \rangle.
\]
Rowley observed that every finite group admitting a fixed-point-free automorphism is solvable \cite[Theorem]{Rowley}. Later, Ladisch proved that every finite group containing an anticentral element is solvable (see \cite{Ladisch2008}). On the other hand, \cite{Ersoy12} provides examples of finite non-solvable groups admitting splitting automorphisms. Indeed, the second author proved in \cite{Ersoy2016} that
a finite group admitting a splitting automorphism of odd order is solvable. In addition, Jabara proved in \cite{Jabara96} that every finite group admitting a splitting automorphism of order \(4\) is solvable.
For nilpotency, the situation is more restrictive.
The existence of splitting automorphisms imposes strong restrictions on the structure of a group. Thompson proved that every finite group admitting a fixed-point-free automorphism of prime order is nilpotent \cite{Thompson59}, while Kegel showed that every finite group admitting a splitting automorphism of prime order is nilpotent \cite{Kegel61}. Moreover, by results of Higman \cite{higman} and Kreknin--Kostrikin \cite{krekkost}, there exists a function \(f\) such that every finite group admitting a fixed-point-free automorphism of prime order \(p\) is nilpotent of class at most \(f(p)\).

Motivated by these results, Sozutov asked the following question.

\begin{problem}{\cite[Sozutov, Problem 10.59]{kourovka21}}\label{prob:sozutov}
Is a \(p'\)-group \(G\) locally nilpotent if it admits a splitting automorphism \(\varphi\) of prime order \(p\) such that all subgroups of the form
\[
\langle g,g^\varphi,\dots,g^{\varphi^{p-1}}\rangle
\]
are nilpotent?
\end{problem}

Our first main result gives a positive answer in the periodic residually finite case, and in fact implies bounded nilpotency class.

\begin{theorem}\label{thm:bounded-class} There is a natural valued function $f$ such that
if \(G\) is a periodic residually finite $p'$-group admitting a splitting automorphism
\(\varphi\) of prime order \(p\)
then \(G\) is nilpotent of class at most \(f(p)\).
\end{theorem}

As an immediate consequence, Sozutov's problem has a positive answer for residually finite groups.

\begin{cor}\label{cor:sozutov-rf}
Let \(G\) be a periodic residually finite \(p'\)-group admitting a splitting automorphism \(\varphi\) of prime order \(p\) such that
\[
\langle g,g^\varphi,\dots,g^{\varphi^{p-1}}\rangle
\]
is nilpotent for every \(g\in G\). Then \(G\) is nilpotent.
\end{cor}
Our second main result shows that a possible counterexample
to Problem \ref{prob:sozutov} cannot be a Tarski monster.

\begin{theorem}\label{no-tarski}
Let \(G\) be a \(p'\)-group admitting a splitting automorphism \(\varphi\) of order \(p\). Assume that
\[
\langle g,g^\varphi,\dots,g^{\varphi^{p-1}}\rangle
\]
is nilpotent for every \(g\in G\). Then \(G\) is not isomorphic to a Tarski monster.
\end{theorem}
One of the most important results in the theory of locally finite groups is the following:
\begin{theorem}[Hall--Kulatilaka]\label{Hall-Kulatilaka}
Every infinite locally finite group contains an infinite abelian subgroup.
\end{theorem}
However, existence of a periodic residually finite analogue of Theorem \ref{Hall-Kulatilaka} is still unknown. In particular, A. Mann asked the following:
\begin{problem}\cite[A. Mann, Problem 11.56]{kourovka21} \label{mannquestion}Does every infinite residually finite group contain an infinite abelian subgroup? Does every infinite residually finite $p$-group contain an infinite abelian subgroup? The same questions are valid for infinite centralizers.
\end{problem}
If there is an infinite residually finite group whose all centralizers are finite, it must contain no involutions by \cite{Shunkov}. Hence, every finite subgroup is solvable. However, since periodic solvable groups are locally finite and locally finite groups have infinite abelian subgroups, in such a counterexample every solvable subgroup is finite. That means if there is a periodic residually finite group with finite centralizers, there is also a non-solvable finitely generated example. Since the possible counterexamples of both of the problems lead us to this direction,
we may formulate an intermediate problem weaker than problems by Sozutov and Mann.
\begin{problem}
Can a hypothetical counterexample to Problem \ref{mannquestion} admit a splitting automorphism?
\end{problem}

\section{Preliminaries}

We collect some basic facts that will be used later.

\begin{prop}\label{prop:centralizer-exp}
Let \(G\) be a group, and let \(\varphi\in \operatorname{Aut}(G)\) be a splitting automorphism of order \(n\). Then \(C_G(\varphi)\) has exponent dividing \(n\).
\end{prop}
\begin{proof}
If \(x\in C_G(\varphi)\), then $$x=x^\varphi=\cdots=x^{\varphi^{n-1}}.$$ Now,
$x.x^\varphi\cdots x^{\varphi^{n-1}}=1$ implies $x^n=1$.
\end{proof}

\begin{prop}\label{prop:fpf-vs-splitting}\cite[7.1.3 Lemma]{khukhro1993nilpotent}
\leavevmode
\begin{enumerate}
\item If \(G\) is finite and \(\varphi\) is fixed-point-free, then \(\varphi\) is a splitting automorphism.
\item If \(G\) is periodic and contains no nontrivial element whose order divides \(n\), then every splitting automorphism of order \(n\) is fixed-point-free.
\end{enumerate}
\end{prop}
The following result shows that on the contrary to finite groups, fixed-point-free automorphisms and splitting automorphisms have different natures.
\begin{theorem}\cite{Neumann1940}
Let \(G\) be a periodic group with a fixed-point-free automorphism of order \(2\). Then \(G\) is abelian without elements of order \(2\).
\end{theorem}
\begin{remark}
    
In infinite groups, in contrast to the finite case, the relationship between fixed-point-free automorphisms and splitting automorphisms is far less direct. For example, the free group of rank \(2\) admits a fixed-point-free automorphism of order \(2\), whereas every group admitting a splitting automorphism of order \(2\) is abelian.
\end{remark}

\begin{prop}\label{prop:quotient}
Let \(G\) be a group, let \(\varphi\in \operatorname{Aut}(G)\) have finite order $n$ and let \(M\trianglelefteq G\) be of finite index. Then there exists a \(\varphi\)-invariant normal subgroup \(N\trianglelefteq G\) of finite index with \(N\leq M\). In particular, \(\varphi\) induces an automorphism of \(G/N\). Moreover, if \(\varphi\) is a splitting automorphism of \(G\), then the induced automorphism on \(G/N\) is also splitting.
\end{prop}
\begin{proof}
Set
\[
N=\bigcap_{i=0}^{n-1}\varphi^i(M).
\]
Since each \(\varphi^i(M)\) is a normal subgroup of finite index in \(G\),
the subgroup \(N\) is also normal and of finite index.

Moreover,
\[
\varphi(N)
=
\bigcap_{i=0}^{n-1}\varphi^{i+1}(M)
=
N,
\]
since \(\varphi^n=1\). Thus \(N\) is \(\varphi\)-invariant, and therefore
\(\varphi\) induces an automorphism \(\overline{\varphi}\) of \(G/N\).

If \(\varphi\) is a splitting automorphism of \(G\), then
\[
g\,g^\varphi\cdots g^{\varphi^{n-1}}=1
\]
for every \(g\in G\). Passing to the quotient modulo \(N\), we obtain
\[
(gN)(gN)^{\overline{\varphi}}\cdots
(gN)^{\overline{\varphi}^{\,n-1}}
=N.
\]
Hence \(\overline{\varphi}\) is a splitting automorphism of \(G/N\).
\end{proof}

\section{Some properties of Tarski monsters}

We now collect the facts about Tarski monsters that will be needed in the sequel.
In particular, all Tarski monsters appearing in this section are periodic, hence infinite, two-generated simple $p$-groups. For the details of their construction, see \cite{Olshanskii1991}.

\begin{prop}\label{prop:tarski-no-automorphism}
Let \(G\) be a Tarski monster. There is no nontrivial automorphism \(\varphi \in \operatorname{Aut}(G)\) of order $n$ such that, for every \(g \in G\), the subgroup
\[
\langle g, g^\varphi, \dots, g^{\varphi^{\,n-1}} \rangle
\]
is nilpotent.
\end{prop}

\begin{proof}
Let \(G\) be a Tarski monster and suppose, for a contradiction, that there exists a non-trivial automorphism \(\varphi \in \operatorname{Aut}(G)\) of order $n$ such that for every \(g \in G\) the subgroup
\[
K_g := \langle g, g^\varphi, \dots, g^{\varphi^{\,n-1}} \rangle
\]
is nilpotent.

Fix \(1 \neq g \in G\). If \(K_g = G\), then \(G\) is nilpotent, which is impossible since $G$ is an infinite simple group.  Hence \(K_g\) is a proper subgroup of \(G\). By definition of a Tarski monster, every proper nontrivial subgroup of \(G\) is cyclic of prime order. Therefore \(K_g = \langle g \rangle\), and in particular \(g^\varphi \in \langle g \rangle\).

Since this holds for every \(g \in G\), it follows that \(\varphi\) is a power automorphism of \(G\).

By \cite{Cooper1968}, every power automorphism of \(G\) is central, that is,
\[
x^{-1}x^\varphi \in Z(G) \quad \text{for all } x \in G.
\]
Since \(Z(G)=1\), it follows that \(x^\varphi = x\) for all \(x \in G\), so \(\varphi = \operatorname{id}_G\), contradicting the assumption that \(\varphi\) is non-trivial.

Therefore no such automorphism exists.
\end{proof}

\begin{remark}\label{rem:involve-analogue-false}
The direct analogue of Proposition \ref{prop:tarski-no-automorphism} for groups merely involving a Tarski monster is false in general; see Example \ref{ex:tarski-times-cyclic}.
\end{remark}

\begin{remark}\label{rem:inner-tarski}
Unlike the finite case, infinite non-solvable groups can have splitting automorphisms of prime order. Let \(G\) be a Tarski monster \(p\)-group. By \cite{Olshanskii1991}, every proper nontrivial subgroup of \(G\) has order \(p\), where \(p\) is a suitably large prime. Fix a nontrivial \(g \in G\). For every \(x \in G\), one has
\[
x x^{g}\cdots x^{g^{p-1}}=(xg^{-1})^p=1.
\]
Thus every inner automorphism of a Tarski monster \(p\)-group induced by a nontrivial element is a splitting automorphism of order \(p\).
\end{remark}

\begin{example}\label{ex:tarski-times-cyclic}
Let \(p\) be a sufficiently large prime, and let \(H\) be a Tarski monster of exponent \(p\). Let \(A\) be a cyclic group of prime order \(q\) such that \(p\mid q-1\). Then \(A\) admits a fixed-point-free automorphism \(\alpha\) of order \(p\). Set \(G=H\times A\), and define
\[
(h,a)^\varphi=(h,a^\alpha)
\qquad (h\in H,\ a\in A).
\]
Then \(\varphi\) is a splitting automorphism of order \(p\) of the infinite finitely generated group \(G\). Moreover, for every \((h,a)\in G\), the subgroup
\[
\langle (h,a),(h,a)^\varphi,\dots,(h,a)^{\varphi^{p-1}}\rangle
\]
is a finite cyclic group. However, \(G\) is not nilpotent.
\end{example}

\section{The proof of the main theorem}

\begin{proof}[Proof of Theorem \ref{thm:bounded-class}] 
Let $c=c(p)$ denote the bound for the nilpotency class of a finite group admitting a fixed-point-free automorphism of order $p$, which exists by \cite{higman}. We claim that
\[
\gamma_{c+1}(G)=1.
\]

Suppose, for a contradiction, that $\gamma_{c+1}(G)\neq 1$, and choose
\[
1\neq x\in \gamma_{c+1}(G).
\]
Since $G$ is residually finite, there exists a normal subgroup $M\trianglelefteq G$ of finite index such that $x\notin M$. Set
\[
N=\bigcap_{i=0}^{p-1}\varphi^{i}(M).
\]
By Proposition~\ref{prop:quotient}, the subgroup $N$ is normal, of finite index, and $\varphi$-invariant. Moreover, $x\notin N$, since $N\leq M$.

Again by Proposition~\ref{prop:quotient}, the automorphism $\varphi$ induces a splitting automorphism $\overline{\varphi}$ on the finite quotient $G/N$. Since $G$ is a $p'$-group, the quotient $G/N$ is also a finite $p'$-group. We now show that $\overline{\varphi}$ is fixed-point-free.

Let $yN\in C_{G/N}(\overline{\varphi})$. Then
\[
(yN)^{\overline{\varphi}}=yN,
\]
and hence
\[
(yN)^{\overline{\varphi}^i}=yN
\qquad \text{for all } i=0,1,\dots,p-1.
\]
Since $\overline{\varphi}$ is a splitting automorphism of $G/N$, we have
\[
(yN)(yN)^{\overline{\varphi}}(yN)^{\overline{\varphi}^2}\cdots (yN)^{\overline{\varphi}^{p-1}}=N.
\]
Therefore,
\[
(yN)^p=N.
\]
As $G/N$ is a finite $p'$-group, it contains no nontrivial element of order dividing $p$. It follows that $yN=N$. Thus
\[
C_{G/N}(\overline{\varphi})=1,
\]
so $\overline{\varphi}$ is fixed-point-free.

By \cite{higman}, the finite group $G/N$ is nilpotent of class at most $c$. Hence
\[
\gamma_{c+1}(G/N)=1,
\]
and therefore
\[
\gamma_{c+1}(G)\leq N.
\]
But $x\in \gamma_{c+1}(G)$ and $x\notin N$, a contradiction. This proves that
\[
\gamma_{c+1}(G)=1.
\]
Consequently, $G$ is nilpotent of class at most $c$, as required.
\end{proof}

Clearly, the proof of Corollary \ref{cor:sozutov-rf} follows immediately from Theorem \ref{thm:bounded-class} and the proof of Theorem \ref{no-tarski} follows from Proposition \ref{prop:tarski-no-automorphism}.

\section*{Acknowledgments}

We gratefully dedicate this paper to the memory of Otto H. Kegel, who guided all three of us with remarkable generosity, wisdom and mathematical vision. His encouragement and example left a deep mark on our work and on our lives as group theorists.

The second author would also like to thank Nesin Mathematics Village, where part of this work was carried out, for its hospitality and excellent research atmosphere.

The first and third authors were supported by UNIPA-FFr 2026. The second author was supported by Next Generation EU (NRRP) project ISP5G+, program SERICS, Investment 7PE00000014- CUP D33C22001300002.

\bibliographystyle{abbrv}
\bibliography{references}

\end{document}